\documentclass[preprint,12pt]{elsarticle}
\usepackage[utf8]{inputenc}
\usepackage{amssymb, amsmath, amsfonts, graphicx, amscd, hyperref,amsthm}

\journal{Topology and its Applications}

\newtheorem{definition}{Definition}[section]

\newtheorem{lemma}[definition]{Lemma}
\newtheorem{theorem}[definition]{Theorem}
\newtheorem{corollary}[definition]{Corollary}

\newdefinition{example}[definition]{Example}
\newdefinition{remark}[definition]{Remark}
\newdefinition{problem}[definition]{Problem}
\newdefinition{question}[definition]{Question}
\newdefinition{fact}[definition]{Fact}
\newproof{pot}{Proof}

\begin{document}

\begin{frontmatter}
\title{A property equivalent to being semi-Kelley}
\author{Mauricio Chac\'on-Tirado\corref{cor1}}
\ead{mauricio.chacon@correo.buap.mx}
\author{Mar\'ia de J. L\'opez}
\ead{mjlopez@fcfm.buap.mx}
\cortext[cor1]{Corresponding author}
\address{Benem\'erita Universidad Aut\'onoma de Puebla, Facultad de Ciencias F\'isico-Matem\'aticas,
Av. San Claudio y 18 sur, Col. San Manuel, Edificio FM1, Ciudad Universitaria C.P. 72570, Puebla, Puebla, M\'exico.}

\author{Ivon Vidal-Escobar}
\ead{piveavis@gmail.com}
%\address{Dirección de UDLAP}

\begin{abstract}
We present a property equivalent to the property of being semi-Kelley. Using this equivalence we prove that being semi-Kelley is a hereditary property for atriodic continua. We prove that semi-Kelley remainders are atriodic, moreover, we prove that semi-Kelley continua are semi-Kelley remainders for chainable continua, circularly chainable continua, and arc continua, and we give an example of an atriodic Kelley continuum which is a semi-Kelley remainder and not a Kelley remainder. We also prove that hereditarily semi-Kelley dendroids are smooth.
\end{abstract}

\begin{keyword}
Atriodic continuum \sep continuum \sep compactification of \((0,1]\) \sep dendroid \sep property of Kelley \sep property of semi-Kelley.
\MSC[2020] Primary \sep 54B20; Secondary \sep 54F15.
\end{keyword}

\end{frontmatter}

\section{Introduction}

A \emph{continuum} is a compact connected metric space with more than one
point. A \emph{subcontinuum} of a continuum \(X\) is a non-empty compact connected
subset of \(X\). For a subset \(A\) of \(X\), let \(cl(A)\) and \(bd(A)\) denote the closure of \(A\) in \(X\) and the boundary of \(A\) in \(X\), respectively.

Given a continuum \(X\), let \(C(X)\) denote the hyperspace of subcontinua of \(X\), endowed with the Hausdorff metric \cite[Definition 2.1]{IN99}.

Recall that a continuum \(X\) is \emph{Kelley at a point \(p\in X\)} provided that
for each subcontinuum \(K\) of \(X\) containing \(p\) and for each sequence \(\{p_{n}\}_{n=1}^{\infty }\) in \(X\)
converging to \(p\), there exists a sequence \(\{K_{n}\}_{n=1}^{\infty }\)
of subcontinua of \(X\) converging to \(K\) (in the Hausdorff metric) such
that \(p_{n}\in K_{n}\), for every \(n\in \mathbb{N}.\) By \cite[Theorem 2.3]{W77}, any continuum is Kelley at each point of a dense \(G_\delta\)-set. A continuum is \emph{Kelley} provided that it is Kelley at each of its points. The property of Kelley was first defined by J. L. Kelley in \cite{K42} to study contractibility of hyperspaces.

Let \(K\) be a subcontinuum of a continuum \(X\). A \emph{maximal limit continuum of }\(K\) is a subcontinuum \(M\subset K\) provided that there is a sequence \(\{M_{n}\}_{n=1}^{\infty }\) of subcontinua of \(X\)
converging to \(M\) such that for each convergent sequence \(\{M_{n}^{\prime }\}_{n=1}^{\infty }\) of subcontinua of \(X\) with \(M_{n}\subset M_{n}^{\prime }\), for each \(n\in \mathbb{N}\) and \(\lim M_{n}^{\prime
}=M^{\prime }\subset K\), we have \(M^{\prime }=M\). A continuum \(X\) is \emph{semi-Kelley} provided that for each subcontinuum
\(K\) of \(X\) and for every two maximal limit continua \(L\) and \(M\)
of \(K\) either \(L\subset M\) or \(M\subset L\). Semi-Kelley continua were introduced by J. J.
Charatonik and W. J. Charatonik in \cite{CC98} as a weaker version
of Kelley continua. An interesting problem in this area is to determine which known results for Kelley continua can
be extended to semi-Kelley continua, see \cite{CENV22}, \cite{C03}, \cite{CC98}, \cite{FP19}, and \cite{I21}.

Besides this introduction, this paper contains four more sections.

\textbf{Section 2}. In this section we state an equivalence to the property of being Kelley and we also state an equivalence to the property of being semi-Kelley. These equivalences give a closer relationship between Kelley continua and semi-Kelley continua.

\textbf{Section 3}. A subcontinuum \(Y\) of a continuum \(X\) is a \emph{neighborhood retract} provided that there exist an open set \(U\) of \(X\) with \(Y\subset U\) and a retraction \(r:U\to Y\). A subcontinuum \(Y\) of \(X\) is called {\it semi-terminal} provided that for each \(A,B\in C(X)\), with \(A\cap B=\emptyset\), \(A\cap Y\neq\emptyset\neq B\cap Y\), we have that \(A\subset Y\) or \(B\subset Y\). For \(n\in\mathbb N\cup\{\infty\}\), a continuum \(X\) is called an {\it \(n\)-od}, provided that \(X\) contains a subcontinuum \(B\), called the {\it core} of \(X\), such that \(X\setminus B\) has at least \(n\) components. A continuum is \emph{atriodic} provided that it contains no 3-ods. In this section we study what kind of subcontinua of semi-Kelley continua are also semi-Kelley. If \(X\) is a semi-Kelley continuum and \(Y\) is a subcontinuum of \(X\), then \(Y\) is semi-Kelley in the following cases: a) \(Y\) is a neighborhood retract of \(X\), b) \(Y\) is semi-terminal, and c) \(X\) is atriodic.

\textbf{Section 4}. A \emph{compactification} of \((0,1]\) with \emph{remainder} \(R\) is a continuum \(X=(0,1]\cup R\), where \((0,1]\) is dense in \(X\) and \(R\) and \((0,1]\) are disjoint.
A \emph{Kelley compactification of \((0,1]\)}, respectively \emph{semi-Kelley compactificacion of \((0,1]\)}, is a Kelley continuum, respectively semi-Kelley continuum, which is also a compactification of \((0,1]\). A continuum \(X\) is a \emph{Kelley remainder}, respectively \emph{semi-Kelley remainder}, provided that it is the remainder of a Kelley compactification of \((0,1]\), respectively semi-Kelley compactification of \((0,1]\). Kelley compactifications were studied in \cite{AI00} and \cite{P05}; Kelley remainders were studied in \cite{BC08} and
\cite{C11}. G. Acosta and A. Illanes showed that if \(X\) is a Kelley
compactification then \(X\) is
atriodic and each subcontinuum of \(X\) is a Kelley continuum
\cite[Theorem 6.2 and Theorem
6.3]{AI00}. P. Pellicer-Covarrubias proved that a continuum \(X\) is hereditarily indecomposable if and only if each
compactification of \((0,1]\) with remainder \(X\) is a
Kelley compactification \cite[Corollary 7.2]{P05}. R. A. Beane and W. J. Charatonik showed that chainable Kelley
continua and Kelley arc continua are Kelley remainders \cite[Theorem 2.3
and Theorem 3.1]{BC08}. M. E. Chac\'{o}n-Tirado proved that circularly chainable Kelley
continua are Kelley remainders \cite[Theorem 1]{C11}. In this section we study what kind of semi-Kelley continua are semi-Kelley remainders. We prove that semi-Kelley remainders are atriodic. We prove that a semi-Kelley continuum is a semi-Kelley remainder for chainable continua, circularly chainable continua, and arc continua. We give an example of an atriodic Kelley continuum which is a semi-Kelley remainder and not a Kelley remainder.

\textbf{Section 5}. A {\it dendroid} is an arc-wise connected
continuum \(X\) such that the intersection of any two subcontinua of \(X\) is
connected. A point \(p\) of a dendroid \(X\) is a \emph{ramification point} provided that there exist three arcs in \(X\) whose pairwise intersection is equal to \(\{p\}\). A \emph{fan} is a dendroid with exactly one ramification point. Given a dendroid \(X\) and \(a,b\in X\), let \(ab\) be the arc in \(X\) with end points \(a\) and \(b\). A dendroid \(X\) is {\it smooth} provided that there exists a point \(p \in X\) such that for each sequence \(\{x_n\}_{n=1}^{\infty}\) in \(X\) converging to some \(x\in X\),
\(\lim px_n = px\). A continuum \(X\) is \emph{hereditarily Kelley}, respectively \emph{hereditarily semi-Kelley}, provided that each subcontinuum of \(X\) is Kelley, respectively semi-Kelley. In this section we prove that hereditarily semi-Kelley dendroids are smooth, compare with \cite[Theorem 6.4]{I21} which states that hereditarily semi-Kelley fans are smooth.

\section{An equivalence to the property of being semi-Kelley}
In this section we introduce an equivalence to the property of being semi-Kelley. Recall that a continuum \(X\) is {\it irreducible between \(p\) and \(q\)} provided that \(p,q\in X\) and no proper subcontinuum of \(X\) contains \(p\) and \(q\). We say that \(X\) is {\it irreducible} if there exist \(p\) and \(q\) such that \(X\) is irreducible between \(p\) and \(q\).

The proof of the following theorem is easy and we leave it to the reader.
\begin{theorem}
	Let \(X\) be a continuum. Then, \(X\) is Kelley if and only if for every \(a\in X\), for every \(I\in C(X)\) irreducible between \(a\) and some other point of \(X\), and for every sequence \(\{a_n\}_{n=1}^\infty\) in \(X\) converging to \(a\), there exists a sequence \(\{A_n\}_{n=1}^\infty\) in \(C(X)\) converging to \(I\) such that \(a_n\in A_n\) for each \(n\in\mathbb N\).
\end{theorem}

We prove the following Theorem.
\begin{theorem}\label{equivalencia}
Let \(X\) be a continuum. Then, \(X\) is semi-Kelley if and only if for every \(a,b\in X\), for every \(I\in C(X)\) irreducible between \(a\) and \(b\), and for every sequence \(\{a_n\}_{n=1}^\infty\) in \(X\) converging to \(a\) there exists a sequence \(\{A_n\}_{n=1}^\infty\) in \(C(X)\) converging to \(I\) such that \(a_n\in A_n\) for each \(n\in\mathbb N\), or for every sequence \(\{b_n\}_{n=1}^\infty\) in \(X\) converging to \(b\) there exists a sequence \(\{B_n\}_{n=1}^\infty\) in \(C(X)\) converging to \(I\) such that \(b_n\in B_n\) for each \(n\in\mathbb N\).
\end{theorem}
\begin{proof}
{\it Necessity}. Suppose that \(X\) is semi-Kelley. Let \(a,b\in X\), let \(I\in C(X)\) be irreducible between \(a\) and \(b\). For a sequence \(\{x_n\}_{n=1}^\infty\) in \(X\) converging to some \(x\in X\), let \(\mathcal M(\{x_n\}_{n=1}^\infty)\) be the family of all \(M\in C(X)\) with \(x\in M\) such that there exists a subsequence \(\{x_{n_k}\}_{k=1}^\infty\) of the sequence \(\{x_n\}_{n=1}^\infty\) and there is a sequence \(\{M_k\}_{k=1}^\infty\) in \(C(X)\) converging to \(M\) such that \(x_{n_k}\in M_k\) for each \(k\in\mathbb{N}\).

\textbf{Claim 1}. For any sequences \(\{x_n\}_{n=1}^\infty\) and \(\{y_n\}_{n=1}^\infty\) in \(X\) converging to \(a\) and \(b\), respectively, there exist a subsequence \(\{x_{n_k}\}_{k=1}^\infty\) of the sequence \(\{x_n\}_{n=1}^\infty\) and a sequence \(\{M_k\}_{k=1}^\infty\) in \(C(X)\) converging to \(I\) such that \(x_{n_k}\in M_k\) for each \(k\in\mathbb{N}\), or there exist a subsequence \(\{y_{n_k}\}_{k=1}^\infty\) of the sequence \(\{y_n\}_{n=1}^\infty\) and a sequence \(\{N_k\}_{k=1}^\infty\) in \(C(X)\) converging to \(I\) such that \(y_{n_k}\in N_k\) for each \(k\in\mathbb{N}\) (compare with \cite[Proposition 3.1]{CC98}).

\emph{Proof of Claim 1.}

By Proposition 3.9 and Statement 3.4 of \cite{CC98}, there exists \(A\in C(I)\cap \mathcal M(\{x_n\}_{n=1}^\infty)\) and there exists \(B\in C(I)\cap \mathcal M(\{y_n\}_{n=1}^\infty)\) such that \(A\) and \(B\) are maximal limit continua of \(I\). Since \(X\) is semi-Kelley, without loss of generality we may assume \(A\subset B\). Notice that \(a\in A\) and \(a,b\in B\subset I\), thus \(B=I\). Hence, there exist a subsequence \(\{y_{n_k}\}_{k=1}^\infty\) of the sequence \(\{y_n\}_{n=1}^\infty\) and a sequence \(\{N_k\}_{k=1}^\infty\) in \(C(X)\) converging to \(I\) such that \(y_{n_k}\in N_k\) for each \(k\in\mathbb{N}\). Claim 1 is proved.

\textbf{Claim 2}. For each open set \(\mathcal U\) of \(C(X)\) with \(I\in \mathcal U\), we have that \(\bigcup\mathcal U\) is a neighborhood of \(a\), or for each open set \(\mathcal V\) of \(C(X)\) with \(I\in \mathcal V\), we have that \(\bigcup\mathcal V\) is a neighborhood of \(b\).

\emph{Proof of Claim 2.}

Assume there exist open sets \(\mathcal U\) and \(\mathcal V\) of \(C(X)\), with \(I\in \mathcal U\cap \mathcal V\) and such that \(\bigcup \mathcal U\) is not a neighborhood of \(a\) and \(\bigcup \mathcal V\) is not a neighborhood of \(b\). Thus, \(\mathcal U\cap \mathcal V\) is an open set of \(C(X)\) and \(\bigcup(\mathcal U\cap \mathcal V)\) is not a neighborhood of \(a\) nor a neighborhood of \(b\). Choose sequences \(\{x_n\}_{n=1}^\infty\) and \(\{y_n\}_{n=1}^\infty\) in \(X\setminus\bigcup(\mathcal U\cap \mathcal V)\) converging to \(a\) and \(b\), respectively, and notice that for no subsequence \(\{x_{n_k}\}_{k=1}^\infty\) of the sequence \(\{x_n\}_{n=1}^\infty\) there exists a sequence \(\{M_k\}_{k=1}^\infty\) in \(C(X)\) converging to \(I\) with \(x_{n_k}\in M_k\) for each \(k\in\mathbb{N}\), and for no subsequence \(\{y_{n_k}\}_{k=1}^\infty\) of the sequence \(\{y_n\}_{n=1}^\infty\) there exists a sequence \(\{N_k\}_{k=1}^\infty\) in \(C(X)\) converging to \(I\) with \(y_{n_k}\in N_k\) for each \(k\in\mathbb{N}\), contradicting Claim 1. Claim 2 is proved.

Without loss of generality, assume that for each open set \(\mathcal U\) of \(C(X)\) with \(I\in \mathcal U\), we have that \(\bigcup\mathcal U\) is a neighborhood of \(a\). Now let \(\{a_n\}_{n=1}^\infty\) be a sequence in \(X\) converging to \(a\).

For each \(m\in\mathbb{N}\), let \(\mathcal U_m=\{A\in C(X):H(A,I)< \frac 1m\}\) and let \(\mathcal U_0=C(X)\). Given \(n\in\mathbb{N}\), choose \(A_n\in C(X)\) in the following way: if \(a_n\in I\), let \(A_n=I\); if \(a_n\notin I\), choose \(m\in\mathbb{N}\) such that \(a_n\in \bigcup \mathcal U_m\setminus\bigcup \mathcal U_{m+1}\) and let \(A_n\in\mathcal U_m\) be such that \(a_n\in A_n\). Notice that the sequence \(\{A_n\}_{n=1}^\infty\) converges to \(I\). This completes the proof of the necessity.

{\it Sufficiency}. Suppose that for every \(a,b\in X\), for every \(I\in C(X)\) irreducible between \(a\) and \(b\), and for every sequence \(\{a_n\}_{n=1}^\infty\) in \(X\) converging to \(a\) there exists a sequence \(\{A_n\}_{n=1}^\infty\) in \(C(X)\) converging to \(I\) such that \(a_n\in A_n\) for each \(n\in\mathbb N\), or for every sequence \(\{b_n\}_{n=1}^\infty\) in \(X\) converging to \(b\) there exists a sequence \(\{B_n\}_{n=1}^\infty\) in \(C(X)\) converging to \(I\) such that \(b_n\in B_n\) for each \(n\in\mathbb N\). Assume that \(X\) is not semi-Kelley, thus there exists \(K\in C(X)\) and there exist maximal limit continua \(L,M\) of \(K\), such that \(L\setminus M\neq\emptyset\neq M\setminus L\). Since \(L\) is a maximal limit continuum of \(K\), there exists a sequence \(\{L_n\}_{n=1}^\infty\) in \(C(X)\) converging to \(L\) such that if \(\{L_n'\}_{n=1}^\infty\) is a sequence in \(C(X)\) converging to some \(L'\in C(K)\) and \(L_n\subset L_n'\) for each \(n\in\mathbb N\), then \(L'=L\). Analogously for \(M\), there exists \(\{M_n'\}_{n=1}^\infty\) its corresponding sequence. Take points \(a\in L\setminus M\), \(b\in M\setminus L\), and sequences \(\{a_n\}_{n=1}^\infty\) and \(\{b_n\}_{n=1}^\infty\) converging to \(a\) and \(b\), respectively, such that \(a_n\in L_n\), \(b_n\in M_n\) for each \(n\in\mathbb N\), and let \(I\in C(K)\) be irreducible between \(a\) and \(b\). Without loss of generality, there exists a sequence \(\{A_n\}_{n=1}^\infty\) in \(C(X)\) converging to \(I\) such that \(a_n\in A_n\) for each \(n\in\mathbb N\). Define the subcontinuum \(L_n'=L_n\cup A_n\) for each \(n\in\mathbb N\), thus \(L_n\subset L_n'\), and the sequence \(\{L_n'\}_{n=1}^\infty\) converges to \(L\cup I\neq L\) and \(L\cup I\in C(K)\). This contradicts to the fact that \(L\) is a maximal limit continuum of \(K\). Thus \(X\) is semi-Kelley. The proof is complete.
\end{proof}

%\begin{corollary}
%Let \(X\) be a continuum. Then \(X\) is semi-Kelley if and only if for every \(Z\in C(X)\), for every \(a,b\in Z\) and for every pair of sequences \(\{a_n\}_{n=1}^\infty\) and \(\{b_n\}_{n=1}^\infty\) converging to \(a\) and \(b\), respectively, there exists \(I\in C(Z)\) (not necessarily irreducible) with \(a,b\in I\) and there exists a sequence \(\{I_n\}_{n=1}^\infty\) in \(C(X)\) converging to \(I\) such that \(a_n\in I_n\) for each \(n\in\mathbb N\) or \(b_n\in I_n\) for each \(n\in\mathbb N\).
%\end{corollary}

\section{Subcontinua having the property of semi-Kelley}
R. Wardle proved that if \(X\) is a Kelley continuum and \(Y\) is a retract of \(X\), then \(Y\) is Kelley, see \cite[Theorem 2.9]{W77}. We prove a similar statement for neighborhood retracts and semi-Kelley continua.

\begin{theorem}
	Let \(X\) be a (semi-)Kelley continuum and let \(Y\in C(X)\) be a neighborhood retract of \(X\). Then, \(Y\) is (semi-)Kelley.
\end{theorem}
\begin{proof}
Suppose that \(X\) is semi-Kelley, the proof when \(X\) is Kelley is similar. Let \(U\) be an open set of \(X\) containing \(Y\) such that there exists a retraction \(\rho:U\to Y\). Let \(a,b\in Y\), let \(I\in C(Y)\) be irreducible between \(a\) and \(b\), and let \(\{a_n\}_{n=1}^\infty\), \(\{b_n\}_{n=1}^\infty\) be sequences in \(Y\) converging to \(a\) and \(b\), respectively. Since \(X\) is semi-Kelley, by Theorem \ref{equivalencia}, without loss of generality assume there exists a sequence \(\{A_n'\}_{n=1}^\infty\) in \(C(X)\) converging to \(I\) such that \(a_n\in A_n'\) for each \(n\in\mathbb N\). Since \(I\subset Y\subset U\), there exists \(N\in\mathbb N\) such that \(A_n'\subset U\) for \(n\geq N\). For each \(n<N\), define \(A_n=\{a_n\}\), and for each \(n\geq N\), define \(A_n=\rho(A_n')\), thus the sequence \(\{A_n\}_{n=1}^\infty\) is contained in \(C(Y)\), converges to \(I\), and \(a_n\in A_n\) for each \(n\in\mathbb N\). By Theorem \ref{equivalencia}, \(Y\) is semi-Kelley.
\end{proof}

\begin{corollary}
	Let \(X\) be a (semi-)Kelley continuum and let \(Y\in C(X)\) be a retract of \(X\). Then, \(Y\) is (semi-)Kelley.
\end{corollary}

J. Prajs proved that if \(X\) is Kelley and \(Y\in C(X)\) is semi-terminal, then \(Y\) is Kelley, see \cite[Proposition 4.2]{P11}. Here we state an analogous statement for semi-Kelley continua and, using Theorem \ref{equivalencia}, the proof is analogous to the proof of Prajs.

\begin{theorem}
Let \(X\) be a semi-Kelley continuum and let \(Y\in C(X)\) be semi-terminal. Then, \(Y\) is semi-Kelley.
\end{theorem}
\begin{proof}
Let \(a,b\in Y\), let \(I\in C(Y)\) be irreducible between \(a\) and \(b\), and let \(\{a_n\}_{n=1}^\infty\), \(\{b_n\}_{n=1}^\infty\) be sequences in \(Y\) converging to \(a\) and \(b\), respectively. By Theorem \ref{equivalencia}, without loss of generality assume that there exists a sequence \(\{A_n\}_{n=1}^\infty\) in \(C(X)\) converging to \(I\) such that \(a_n\in A_n\) for each \(n\in\mathbb N\). If \(A_n\subset Y\) for each \(n\in\mathbb N\), Theorem \ref{equivalencia} finishes the proof that \(Y\) is semi-Kelley. Thus we may assume that \(A_n\setminus Y\neq\emptyset\) for each \(n\in\mathbb N\). By \cite[Proposition 3.1]{P11}, the sets \(L_n=A_n\cap Y\) are continua, and by the same arguments as in the proof of \cite[Proposition 4.2]{P11}, the sets \(K_n=L_n\cup I\) are subcontinua of \(Y\) converging to \(I\) and \(a_n\in K_n\) for each \(n\in\mathbb N\). By Theorem \ref{equivalencia}, \(Y\) is semi-Kelley.
\end{proof}

\begin{theorem}\label{infinitoodo}
Let \(X\) be a semi-Kelley continuum but not hereditarily semi-Kelley. Then, \(X\) contains an \(\infty\)-od.
\end{theorem}
\begin{proof}
Let \(A\) be a subcontinuum of \(X\) without the property of semi-Kelley. By Theorem \ref{equivalencia}, there exist \(a,b\in A\), there exists \(K\in C(A)\) irreducible between \(a\) and \(b\), there exist sequences \(\{a_n\}_{n=1}^\infty\) and \(\{b_n\}_{n=1}^\infty\) in \(A\) converging to \(a\) and \(b\), respectively, such that there is no sequence \(\{L_n\}_{n=1}^\infty\) in \(C(A)\) converging to \(K\) with \(a_n\in L_n\) for each \(n\in\mathbb N\), and there is no sequence \(\{M_n\}_{n=1}^\infty\) in \(C(A)\) converging to \(K\) with \(b_n\in M_n\) for each \(n\in\mathbb N\). By Theorem \ref{equivalencia}, without loss of generality, assume there exists a sequence \(\{K_n\}_{n=1}^\infty\) in \(C(X)\) converging to \(K\) such that \(a_n\in K_n\) for each \(n\in\mathbb N\). Let \(C_n\) be the component of \(A\cap K_n\) containing \(a_n\) for each \(n\in\mathbb N\). Following the proof of \cite[Theorem 5.1]{AI00}, we can find an increasing sequence \(\{n_j\}_{j=1}^\infty\) in \(\mathbb N\) and a sequence \(\{D_j\}_{j=1}^\infty\) in \(C(X)\) with \(C_{n_j}\subsetneq D_j\subset K_{n_j}\) for each \(j\in\mathbb N\) and such that the set \(D=A\cup(\bigcup_{j=1}^\infty D_j)\) is an \(\infty\)-od. The proof is complete.
\end{proof}

G. Acosta and A. Illanes proved that atriodic Kelley continua are hereditarily Kelley, see \cite[Corollary 5.2]{AI00}. As a corollary of Theorem \ref{infinitoodo}, we obtain an analogous statement for semi-Kelley continua.
\begin{corollary}
Let \(X\) be an atriodic semi-Kelley continuum. Then, \(X\) is hereditarily semi-Kelley.
\end{corollary}

\section{Compactifications of \((0,1]\) having the property of semi-Kelley}

In this section we prove that semi-Kelley remainders are atriodic. Moreover we prove that a semi-Kelley continuum is a semi-Kelley remainder for chainable continua, circularly chainable continua, and arc continua. We give an example of an atriodic Kelley continuum which is a semi-Kelley remainder and not a Kelley remainder.

We will use a special case of Theorem \ref{equivalencia}.
\begin{lemma}\label{equivalenciaCompactacion}
Let \(X\) be a continuum and let \(Y=X\cup(0,1]\) be a compactification of \((0,1]\) with remainder \(X\). Then, \(Y\) is semi-Kelley if and only if for each \(a,b\in X\), for each \(I\in C(X)\) irreducible proper subcontinuum between \(a\) and \(b\), and for each sequence \(\{a_n\}_{n=1}^\infty\) in \(Y\) converging to \(a\) there exists a sequence \(\{A_n\}_{n=1}^\infty\) in \(C(Y)\) converging to \(I\) such that \(a_n\in A_n\subset X\) or \(a_n\in A_n\subset(0,1]\) for each \(n\in\mathbb N\), or for each sequence \(\{b_n\}_{n=1}^\infty\) in \(Y\) converging to \(b\) there exists a sequence \(\{B_n\}_{n=1}^\infty\) in \(C(Y)\) converging to \(I\) such that \(b_n\in B_n\subset X\) or \(b_n\in B_n\subset(0,1]\) for each \(n\in\mathbb N\).
\end{lemma}
\begin{proof}
{\it Necessity}. Assume that \(Y\) is a semi-Kelley. Let \(a,b\in X\), let \(I\in C(X)\) be an irreducible proper subcontinuum between \(a\) and \(b\), and let \(\{a_n\}_{n=1}^\infty\), \(\{b_n\}_{n=1}^\infty\) be sequences in \(Y\) converging to \(a\) and \(b\), respectively. By Theorem \ref{equivalencia}, without loss of generality there exists a sequence \(\{A_n'\}_{n=1}^\infty\) in \(C(Y)\) converging to \(I\) with \(a_n\in A_n'\) for each \(n\in\mathbb N\). Since \(X\) is a terminal subcontinuum of \(Y\) and \(\{A_n'\}_{n=1}^\infty\) converges to \(I\), there exists \(M\in\mathbb N\) such that for each \(n\geq M\), \(A_n'\subset X\) or \(A_n'\cap X=\emptyset\). For each \(n<M\), define \(A_n=\{a_n\}\) and for each \(n\geq M\), define \(A_n=A_n'\), thus the sequence \(\{A_n\}_{n=1}^\infty\) converges to \(I\) and \(a_n\in A_n\subset X\) or \(a_n\in A_n\subset(0,1]\) for each \(n\in\mathbb N\). This completes the proof of the necessity.

{\it Sufficiency}. Assume for each \(a,b\in X\), for each \(I\in C(X)\) irreducible proper subcontinuum between \(a\) and \(b\), and for each sequence \(\{a_n\}_{n=1}^\infty\) in \(Y\) converging to \(a\) there exists a sequence \(\{A_n\}_{n=1}^\infty\) in \(C(Y)\) converging to \(I\) such that \(a_n\in A_n\subset X\) or \(a_n\in A_n\subset(0,1]\) for each \(n\in\mathbb N\), or for each sequence \(\{b_n\}_{n=1}^\infty\) in \(Y\) converging to \(b\) there exists a sequence \(\{B_n\}_{n=1}^\infty\) in \(C(Y)\) converging to \(I\) such that \(b_n\in B_n\subset X\) or \(b_n\in B_n\subset(0,1]\) for each \(n\in\mathbb N\). We will prove that \(Y\) is semi- Kelley. Let \(a,b\in Y\), let \(I\in C(Y)\) be irreducible between \(a\) and \(b\), and let \(\{a_n\}_{n=1}^\infty\), \(\{b_n\}_{n=1}^\infty\) be sequences in \(Y\) converging to \(a\) and \(b\), respectively. If \(a\in(0,1]\) or \(b\in(0,1]\), then \(Y\) has the property of Kelley at \(a\) or \(b\), thus there exists a sequence \(\{A_n\}_{n=1}^\infty\) in \(C(Y)\) converging to \(I\) with \(a_n\in A_n\) for each \(n\in\mathbb N\), or there exists a sequence \(\{B_n\}_{n=1}^\infty\) in \(C(Y)\) converging to \(I\) with \(b_n\in B_n\) for each \(n\in\mathbb N\). Now, suppose that \(a,b\in X\), observe that \(I\in C(X)\); then by hypothesis there exists a sequence \(\{A_n\}_{n=1}^\infty\) in \(C(Y)\) converging to \(I\) such that \(a_n\in A_n\subset X\) or \(a_n\in A_n\subset(0,1]\) for each \(n\in\mathbb N\), or there exists a sequence \(\{B_n\}_{n=1}^\infty\) in \(C(Y)\) converging to \(I\) such that \(b_n\in B_n\subset X\) or \(b_n\in B_n\subset(0,1]\) for each \(n\in\mathbb N\). By Theorem \ref{equivalencia}, \(Y\) is semi-Kelley. The proof is complete.
\end{proof}

Acosta and Illanes proved that Kelley remainders are atriodic, see \cite[Theorem 6.2]{AI00}. We prove the analogous result for semi-Kelley continua.
\begin{theorem}
Let \(X\) be a semi-Kelley remainder. Then, \(X\) is atriodic.
\end{theorem}
\begin{proof}
Let \(Y=X\cup(0,1]\) be a semi-Kelley compactification of \((0,1]\) with remainder \(X\).
Assume \(T\subset X\) is a 3-od with core \(H\). Let \(J,K,L\) be distinct components of \(T\setminus H\); by \cite[Corollary 5.9]{N92}, \(H\cup J\), \(H\cup K\), and \(H\cup L\) are subcontinua of \(T\). Let \(a\in J,\) \(b\in K\), and \(c\in L\). Let \(\{a_n\}_{n=1}^\infty\), \(\{b_n\}_{n=1}^\infty\), and \(\{c_n\}_{n=1}^\infty\) be sequences in \((0,1]\) converging to \(a,\) \(b\), and \(c\), respectively.

Let \(I(a,c)\) be irreducible between \(a\) and \(c\) with \(I(a,c)\subset H\cup J\cup L\). By Lemma \ref{equivalenciaCompactacion}, since \(Y\) is semi-Kelley, without loss of generality we may assume that there exists a sequence \(\{A_n'\}_{n=1}^\infty\) in \(C(Y)\) converging to \(I(a,c)\) with \(a_n\in A_n'\subset(0,1]\) for each \(n\in\mathbb N\). Let \(\{d_n\}_{n=1}^\infty\) be a sequence converging to \(c\) with \(d_n\in A_n'\) for each \(n\in\mathbb N\), and let \(A_n\subset A_n'\) be the arc joining \(a_n\) and \(d_n\), notice that \(\{A_n\}_{n=1}^\infty\) converges to \(I(a,c)\).

Let \(I(b,c)\) be irreducible between \(b\) and \(c\) with \(I(b,c)\subset H\cup K\cup L\). Similarly as above and without loss of generality, there exists a sequence \(\{e_n\}_{n=1}^\infty\) in \((0,1]\) converging to \(c\) and there exists a sequence \(\{B_n\}_{n=1}^\infty\) of subcontinua of \((0,1]\) converging to \(I(b,c)\), where for each \(n\in\mathbb N\), \(B_n\) is the arc in \((0,1]\) joining \(b_n\) and \(e_n\).

Let \(J'\) be a subcontinuum of \(H\cup J\) with \(H\subsetneq J'\), \(a\notin J'\), and let \(K'\) be a subcontinuum of \(H\cup K\) with \(H\subsetneq K'\), \(b\notin K'\). Let \(a'\in (I(a,c)\cap J')\setminus H\), let \(b'\in (I(b,c)\cap K')\setminus H\) and let \(I(a',b')\) be irreducible between \(a'\) and \(b'\) with \(I(a',b')\subset J'\cup K'\). Notice \(\{a,b,c\}\cap I(a',b')=\emptyset\), \(b'\notin I(a,c)\) and \(a'\notin I(b,c)\). Since \(\{A_n\}_{n=1}^\infty\) converges to \(I(a,c)\) and \(a'\in I(a,c)\), we can choose a sequence \(\{a_n'\}_{n=1}^\infty\) converging to \(a'\) with \(a_n'\in A_n\) for each \(n\in\mathbb N\); analogously choose a sequence \(\{b_n'\}_{n=1}^\infty\) converging to \(b'\) with \(b_n'\in B_n\) for each \(n\in\mathbb N\). By Lemma \ref{equivalenciaCompactacion}, since \(Y\) is semi-Kelley, without loss of generality assume there exists a sequence \(\{M_n\}_{n=1}^\infty\) of subcontinua of \((0,1]\) converging to \(I(a',b')\) with \(a_n'\in M_n\) for each \(n\in\mathbb N\). Since \(b'\notin I(a,c)\) and \(\{A_n\}_{n=1}^\infty\) converges to \(I(a,c)\), then \(M_n\not\subset A_n\) for \(n\) sufficiently large. Recall that \(A_n\) is an arc with \(a_n\) and \(d_n\) as its end points. Hence, there exist \(S\subset\mathbb N\) infinite with \(a_n\in M_n\) for each \(n\in S\) or \(d_n\in M_n\) for each \(n\in S\), in the first case we have that \(a\in I(a',b')\), and in the second case we have that \(c\in I(a',b')\), which is a contradiction. The proof is complete.
\end{proof}

For a continuum \(X\), a {\it chain} in \(X\) is a non-empty, finite collection \(\mathcal C=\{U_1,\dots,U_n\}\) of non-empty open sets \(U_i\) of \(X\) for each \(i\in \{1,\dots,n\}\), such that \(U_i\cap U_j\neq\emptyset\) if and only if \(|i-j|\leq 1\), the members of a chain \(\mathcal C\) are called {\it links} of \(\mathcal C\). For \(\varepsilon>0\), a chain \(\mathcal C\) is
an {\it \(\varepsilon\)-chain} if mesh\((\mathcal C)<\varepsilon\), where \({\text {mesh}}(\mathcal{C})=\max\{{\text {diam}}(U_i):1\leq i\leq n\}\), and diam\((U_i)\) is the diameter of \(U_i\). A continuum \(X\) is {\it chainable} provided that for each \(\varepsilon>0\), there is an \(\varepsilon\)-chain covering to \(X\).
\begin{theorem}\label{chainable}
Let \(X\) be a semi-Kelley chainable continuum. Then, \(X\) is a semi-Kelley remainder.
\end{theorem}
\begin{proof}
By \cite[Theorem 5]{B62}, suppose that \(X\) is embedded in \(\mathbb R^2\) in such a way that for each \(\varepsilon>0\), \(X\) can be covered by chains of mesh less than \(\varepsilon\) and whose links are open balls. Thus, we assume \(X\subset\mathbb R^2\times\{0\}\subset\mathbb R^3\). Let \(d\) be the Euclidean metric in \(\mathbb{R}^3\).

Given \(n\in\mathbb N\), fix a chain \(\mathcal C_n=\{U_1^n,\dots,U_{r(n)}^n\}\) of open balls in \(\mathbb R^3\), such that \(\mathcal C_n\) covers \(X\) and mesh(\(\mathcal C_n\))\( <\frac 1n\), we suppose \(r(n)>2\). For each \(Z\in C(X)\), let \(\mathcal C_n(Z)=\{U\in \mathcal C_n:U\cap Z\neq\emptyset\}\).

For each \(i\in\{1,\dots,r(n)-1\}\), we fix a point \((x_i^n,0)\in U_i^n\cap U_{i+1}^n\cap X\). Let \(x_i^nx_{i+1}^n\) be the convex segment in \(\mathbb R^2\) joining \(x_i^n\) and \(x_{i+1}^n\). Let \(R_n=\bigcup\{x_i^nx_{i+1}^n\times\{\frac 1n\}:i\in\{1,\dots,r(n)-2\}\}\). Notice that \(R_n\) is an arc and, since mesh(\(\mathcal C_n\))\(<\frac 1n\), \(\lim R_n=X\). Taking a subsequence if necessary, we may assume that the sequence \(\{(x_1^n,0)\}_{n=1}^\infty\) converges to a point \((x,0)\in X\), that the sequence \(\{(x_{r(n)-1}^n,0)\}_{n=1}^\infty\) converges to a point \((y,0)\in X\), and if \(n>1\), we have \((x_1^n,0)\in U_1^{n-1}\) and \((x_{r(n)-1}^n,0)\in U_{r(n-1)}^{n-1}\). Let \(S_n\) be the convex segment in \(\mathbb R^3\) joining \((x_{r(n)-1}^n,\frac 1n)\) and \((x_{r(n+1)-1}^{n+1},\frac 1{n+1})\) if \(n\) is odd, and joining \((x_{1}^n,\frac 1n)\) and \((x_{1}^{n+1},\frac 1{n+1})\) if \(n\) is even. Notice that \(\lim S_{2n}=\{(x,0)\}\) and \(\lim S_{2n+1}=\{(y,0)\}\).

Let \(R=\bigcup\{R_n\cup S_n:n\in\mathbb N\}\), let \(Y=X\cup R\) and notice \(R\) is homeomorphic to \((0,1]\) and \(Y\) is a compactificacion of \(R\) with remainder \(X\).

Now, we prove that \(Y\) is semi-Kelley using Lemma \ref{equivalenciaCompactacion}. Let \(a,b\in X\), let \(I\in C(X)\) be irreducible proper subcontinuum between \(a\) and \(b\), and let \(\{a_n\}_{n=1}^\infty\), \(\{b_n\}_{n=1}^\infty\) be sequences in \(Y\) converging to \(a\) and \(b\), respectively. We define two sequences \(\{a_n'\}_{n=1}^\infty\) and \(\{b_n'\}_{n=1}^\infty\) in \(X\) in the following way: Given \(n\in\mathbb{N}\), if \(a_n\in X\), we let \(a_n'=a_n\); if \(a_n\in R_{m_n}\) for some \(m_n\in\mathbb{N}\), choose \(i\in\{1,\dots,r(m_n)-2\}\) such that \(a_n\in x_i^{m_n}x_{i+1}^{m_n}\times\{\frac{1}{m_n}\}\) and we let \(a_n'=(x_i^{m_n},0)\); if \(a_n\in S_{m_n}\) for some \(m_n\in\mathbb{N}\), we let \(a_n'=(x_{r(m_n)-1}^{m_n},0)\) if \(m_n\) is odd, and \(a_n'=(x_{1}^{m_n},0)\) if \(m_n\) is even; analogously we define the sequence \(\{b_n'\}_{n=1}^\infty\). Since \(d(a_n,a_n')<\frac{2}{m_n}\) for each \(a_n\in R\) and \(\lim m_n=\infty\), then \(\{a_n'\}_{n=1}^\infty\) converges to \(a\); analogously \(\{b_n'\}_{n=1}^\infty\) converges to \(b\).

By Theorem \ref{equivalencia}, since \(X\) is semi-Kelley, without loss of generality there exists a sequence \(\{A_n'\}_{n=1}^\infty\) in \(C(X)\) converging to \(I\) such that \(a_n'\in A_n'\) for each \(n\in\mathbb N\).

Now, we will construct a sequence \(\{A_n\}_{n=1}^\infty\) in \(C(Y)\) converging to \(I\) such that \(a_n\in A_n\subset X\) or \(a_n\in A_n\subset R\) for each \(n\in\mathbb N\). Given \(n\in\mathbb{N}\), if \(a_n\in X\), we let \(A_n=A_n'\); if \(a_n\in R_{m_n}\), let \(A_n=\bigcup\{x_i^{m_n}x_{i+1}^{m_n}\times\{\frac{1}{m_n}\}:i\in\{1,\dots,r_{m_n}-2\}\text{ and }U_i^{m_n}\in \mathcal C_{m_n}(A_n')\}\); if \(a_n\in S_{m_n}\), let \(A_n=S_{m_n}\cup \bigcup\{x_i^{m_n}x_{i+1}^{m_n}\times\{\frac{1}{m_n}\}:i\in\{1,\dots,r_{m_n}-2\}\text{ and }U_i^{m_n}\in \mathcal C_{m_n}(A_n')\}\).

Notice \(a_n\in A_n\subset X\) or \(a_n\in A_n\subset R\) for each \(n\in\mathbb N\), and \(A_n\) is a subcontinuum of \(Y\). Since \(H(A_n',A_n)<\frac{2}{m_n}\) for each \(a_n\in R\) and \(\lim m_n=\infty\), then the sequence \(\{A_n\}_{n=1}^\infty\) converges to \(I\). By Lemma \ref{equivalenciaCompactacion}, \(Y\) is semi-Kelley. This finishes the proof.
\end{proof}

A {\it circular chain} in \(X\) is a non-empty, finite collection \(\mathcal C=\{U_1,\dots,U_n\}\) of non-empty open sets \(U_i\) of \(X\) for each \(i\in\{1,\dots,n\}\), such that \(U_i\cap U_j\neq\emptyset\) if and only if \(|i-j|\leq 1\) or \(\{i,j\}=\{1,n\}\). For \(\varepsilon>0\), a circular chain \(\mathcal C\) is an {\it \(\varepsilon\)-circular chain} if mesh\((\mathcal C)<\varepsilon\). A continuum \(X\) is {\it circularly chainable} provided that for each \(\varepsilon>0\), there is an \(\varepsilon\)-circular chain covering to \(X\).
\begin{theorem}\label{circle}
Let \(X\) be a semi-Kelley circularly chainable continuum. Then, \(X\) is a semi-Kelley remainder.
\end{theorem}
\begin{proof}
Using a similar construction as done in \cite[Theorem 4]{B62}, suppose that \(X\) is embedded in \(\mathbb R^3\) in such a way that for each \(\varepsilon>0\), \(X\) can be covered by circular chains of mesh less than \(\varepsilon\), and whose links are open balls. Thus, we assume \(X\subset\mathbb R^3\times\{0\}\subset\mathbb R^4\). Let \(d\) be the Euclidean metric in \(\mathbb{R}^4\).
	
Fix \(z\in X\) such that \(X\) is Kelley at \(z\). Given \(n\in\mathbb N\), fix a circular chain \(\mathcal C_n=\{U_1^n,\dots,U_{r(n)}^n\}\) of open balls in \(\mathbb R^4\), such that \(\mathcal C_n\) covers \(X\) and mesh(\(\mathcal C_n\))\( <\frac 1n\), moreover we suppose \(r(n)>3\) and \(z\in U_1^n\setminus U_2^n\). For each \(Z\in C(X)\), let \(\mathcal C_n(Z)=\{U\in \mathcal C_n:U\cap Z\neq\emptyset\}\).

For each \(i\in\{1,\dots,r(n)\}\), we fix a point \((x_i^n,0)\in U_i^n\cap U_{i\oplus_{r(n)}1}^n\cap X\), where \(\oplus_{r(n)}\) is the sum module \(r(n)\). For \(x,y\in\mathbb R^3\), let \(xy\) be the convex segment in \(\mathbb R^3\) joining \(x\) and \(y\). Let \(R_n=\bigcup\{x_i^nx_{i+1}^n\times\{\frac 1n\}:i\in\{1,\dots,r(n)-1\}\}\) and let \(S_n=R_n\cup (x_{r(n)}^nx_{1}^n\times\{\frac 1n\})\). Notice that \(S_n\) is a circle, \(R_n\) is an arc in \(S_n\), since mesh(\(\mathcal C_n\))\(<\frac 1n\), we have \(\lim S_n=X=\lim R_n\) and \(\lim x_{r(n)}^nx_{1}^n\times\{\frac 1n\}=\{z\}\).

Proceeding as in \cite[p. 172]{C11}, we can choose an increasing sequence \(\{p_k\}_{k=1}^{\infty}\) of natural numbers with the following properties (For each \(k>1\), let \(P_k\) be the convex segment in \(\mathbb{R}^4\) joining \((x_{r(p_{k-1})}^{p_{k-1}},\frac{1}{p_{k-1}})\) and \((x_{1}^{p_{k}},\frac{1}{p_{k}})\) and let \(P_1=\{(x_1^{p_1},\frac{1}{p_1})\}\)):

(P1) the set \(Y=X\cup\bigcup\{R_{p_k}\cup P_k:k\in\mathbb{N}\}\) is a compactification of the ray \(R=\bigcup\{R_{p_k}\cup P_k:k\in\mathbb{N}\}\), with remainder \(X\), and

(P2) for each \(Z\in C(X)\) with \(z\in Z\) there exists a sequence \(\{Z_k\}_{k=1}^\infty\) of subcontinua of \(Y\) converging to \(Z\) with \(P_k\subset Z_k\subset R\) for each \(k\in\mathbb N\).

Now, we prove that \(Y\) is semi-Kelley using Lemma \ref{equivalenciaCompactacion}. Let \(a,b\in X\), let \(I\in C(X)\) be irreducible proper subcontinuum between \(a\) and \(b\), and let \(\{a_n\}_{n=1}^\infty\), \(\{b_n\}_{n=1}^\infty\) be sequences in \(Y\) converging to \(a\) and \(b\), respectively. We define two sequences \(\{a_n'\}_{n=1}^\infty\) and \(\{b_n'\}_{n=1}^\infty\) in \(X\) in the following way: Given \(n\in\mathbb{N}\), if \(a_n\in X\), we let \(a_n'=a_n\); if \(a_n\in R_{m_n}\) for some \(m_n\in\mathbb{N}\), choose \(i\in\{1,\dots,r(m_n)-1\}\) such that \(a_n\in x_i^{m_n}x_{i+1}^{m_n}\times\{\frac{1}{m_n}\}\) and we let \(a_n'=(x_i^{m_n},0)\); if \(a_n\in P_{m_n}\) for some \(m_n\in\mathbb{N}\), we let \(a_n'=(x_{1}^{p_{m_n}},0)\). Analogously we define the sequence \(\{b_n'\}_{n=1}^\infty\). Since \(d(a_n,a_n')<\frac{2}{m_n}\) for each \(a_n\in R\) and \(\lim m_n=\infty\), then \(\{a_n'\}_{n=1}^\infty\) converges to \(a\); analogously \(\{b_n'\}_{n=1}^\infty\) converges to \(b\).

By Theorem \ref{equivalencia}, without loss of generality, assume there exists a sequence \(\{A_n'\}_{n=1}^\infty\) in \(C(X)\) converging to \(I\) such that \(a_n'\in A_n'\) for each \(n\in\mathbb N\). We will construct a sequence \(\{A_n\}_{n=1}^\infty\) in \(C(Y)\) converging to \(I\) such that \(a_n\in A_n\subset X\) or \(a_n\in A_n\subset R\) for each \(n\in\mathbb N\). We consider two cases:

Case 1. The subcontinuum \(I\) does not contain \(z\). Given \(n\in\mathbb{N}\), if \(a_n\in X\), we let \(A_n=A_n'\); if \(a_n\in \bigcup_{k=1}^\infty P_k\), we let \(A_n=\{a_n\}\), notice that, since \(\lim a_n=a\in I\), \(\lim P_k=\{z\}\), and \(z\notin I\), we have \(a_n\in \bigcup_{k=1}^\infty P_k\) only for a finite number of \(n\in\mathbb{N}\); if \(a_n\in R_{m_n}\) for some \(m_n\in\mathbb{N}\), notice that \(\mathcal C_{m_n}(A_n')\) is a subchain of \(\mathcal C_{m_n}\) that does not contain the links \(C_{m_n}^1\) and \(C_{m_n}^{r(m_n)}\), for sufficiently large \(n\), and let \(A_n=\bigcup\{x_i^{m_n}x_{i+1}^{m_n}\times\{\frac{1}{m_n}\}:i\in\{1,\dots,r_{m_n}-1\}\text{ and }U_i^{m_n}\in \mathcal C_{m_n}(A_n')\}\). Then, \(\{A_n\}_{n=1}^\infty\) is a sequence in \(C(Y)\) converging to \(I\) and satisfies \(a_n\in A_n\subset X\) or \(a_n\in A_n\subset R\) for each \(n\in\mathbb N\).

Case 2. The subcontinuum \(I\) contains \(z\). By (P2), let \(\{Z_k\}_{k=1}^\infty\) be a sequence in \(C(Y)\) converging to \(I\) such that \(P_k\subset Z_k\subset R\), for each \(k\in\mathbb N\). Given \(n\in\mathbb{N}\), if \(a_n\in X\), we let \(A_n=A_n'\); if \(a_n\in P_{m_n}\) for some \(m_n\in\mathbb{N}\), we let \(A_n=Z_{m_n}\); if \(a_n\in R_{m_n}\) for some \(m_n\in\mathbb{N}\), assume \(a'_n=(x_{s_n}^{m_n},0)\) for some \(s_n\in\{1,2,\dots,r(m_n)-1\}\), we consider two cases:

Case a. The subchain \(\mathcal C_{m_n}(A_n')\) does not contain the link \(C_{m_n}^1\) or does not contain the link \(C_{m_n}^{r(m_n)}\). Let \(A_n=\bigcup\{x_i^{m_n}x_{i+1}^{m_n}\times\{\frac{1}{m_n}\}:i\in\{1,\dots,r_{m_n}-1\}\text{ and }U_i^{m_n}\in \mathcal C_{m_n}(A_n')\}\).

Case b. The subchain \(\mathcal C_{m_n}(A_n')\) contains the links \(C_{m_n}^1\) and \(C_{m_n}^{r(m_n)}\). Define \(J_n=\bigcup\{x_i^{m_n}x_{i+1}^{m_n}\times\{\frac{1}{m_n}\}:i\in\{1,\dots,s_n\}\text{ and }U_i^{m_n}\in \mathcal C_{m_n}(A_n')\}\)
and \(L_n=\bigcup\{x_i^{m_n}x_{i+1}^{m_n}\times\{\frac{1}{m_n}\}:i\in\{s_n,\dots,r_{m_n}-1\}\text{ and }U_i^{m_n}\in \mathcal C_{m_n}(A_n')\}\). Since \(\mathcal C_{m_n}(A_n')\) is a subchain, then \(J_n\) is connected or \(L_n\) is connected, moreover, \(\limsup J_n\subset I\) and \(\limsup L_n\subset I\). If \(J_n\) is connected, we let \(A_n=J_n\cup Z_{n}\); if \(L_n\) is connected, we let \(A_n=L_n\cup Z_{n+1}\).

Then \(\{A_n\}_{n=1}^\infty\) is a sequence in \(C(Y)\) converging to \(I\) and satisfies \(a_n\in A_n\subset X\) or \(a_n\in A_n\subset R\) for each \(n\in\mathbb N\). By Lemma \ref{equivalenciaCompactacion}, \(Y\) is semi-Kelley. This finishes the proof.
\end{proof}

A continuum is \emph{decomposable} provided that it is the union of two proper subcontinua, and it is \emph{indecomposable} provided that it is not decomposable.  An \emph{arc continuum} is a continuum whose proper subcontinua are arcs. Given a continuum \(X\), a Whitney map for \(C(X)\)
is a mapping \(\omega:C(X)\to [0,1)\) such that \(\omega(A)<\omega(B)\) for any \(A, B\in C(X)\) with \(A\subsetneq B\), and \(\omega(A)=0\) if and only if \(A\) is a singleton. It is known that for any continuum \(X\), Whitney maps exist for \(C(X)\) \cite[Theorem 13.4]{IN99}.

\begin{theorem}
Let \(X\) be a semi-Kelley arc continuum. Then, \(X\) is a semi-Kelley remainder.
\end{theorem}
\begin{proof}
Let \(X\) be a semi-Kelley arc continuum. If \(X\) is decomposable, then \(X\) is an arc or a simple closed curve. If \(X\) is an arc the result follows from Theorem \ref{chainable}; if \(X\) is a simple closed curve we obtain the result from Theorem \ref{circle}. Now assume \(X\) is indecomposable. Fix \(z\in X\) and we let \(S\subset X\) be a continuous injective image of \([0,\infty)\) such that \(z\) is the image of \(0\) and \(cl(S)=X\). Let \(\omega:C(X)\to [0,1]\) be a Whitney map such that \(\omega(X)=1\), and define \(R=\{(x,1-\omega(zx))\in X\times [0,1]:x\in S\}\). Then, \(Y=R\cup (X\times\{0\})\) is a compactification of the ray \(R\) with remainder \(X\times \{0\}\).

Now, we prove that \(Y\) is semi-Kelley using Lemma \ref{equivalenciaCompactacion}. Let \(a,b\in X\times\{0\}\), let \(I\) be an irreducible proper subcontinuum of \(X\times\{0\}\) between \(a\) and \(b\), and let \(\{a_n\}_{n=1}^\infty\), \(\{b_n\}_{n=1}^\infty\) be sequences in \(Y\) converging to \(a\) and \(b\), respectively. We define two sequences \(\{a_n'\}_{n=1}^\infty\) and \(\{b_n'\}_{n=1}^\infty\) in \(X\times \{0\}\) in the following way: Consider \(\pi:Y\to X\times\{0\}\) the projection map, and let \(a_n'=\pi(a_n)\), \(b_n'=\pi(b_n)\) for each \(n\in\mathbb N\), then \(\{a_n'\}_{n=1}^\infty\) converges to \(a\) and \(\{b_n'\}_{n=1}^\infty\) converges to \(b\).

By Theorem \ref{equivalencia}, without loss of generality, assume there exists a sequence \(\{A_n'\}_{n=1}^\infty\) in \(C(X\times\{0\})\) converging to \(I\) such that \(a_n'\in A_n'\subsetneq X\times\{0\}\) for each \(n\in\mathbb N\). Now define \(A_n\) as the component of \(\pi^{-1}(A'_n)\) containing \(a_n\). Then, \(\{A_n\}_{n=1}^\infty\) is a sequence in \(C(Y)\) converging to \(I\) and satisfies \(a_n\in A_n\subset X\) or \(a_n\in A_n\subset R\) for each \(n\in\mathbb N\). By Lemma \ref{equivalenciaCompactacion}, \(Y\) is semi-Kelley. This finishes the proof.
\end{proof}

\begin{example}
	There exists an atriodic Kelley continuum which is a semi-Kelley remainder and not a Kelley remainder.
\end{example}
\begin{proof}
In defining our continuum, we will use cylindrical coordinates \((r,\theta,z)\). Define \(S^1=\{(r,\theta,0):r=1,\theta\in[0,2\pi]\}\), \(R=\{(r,\theta,0):r=1+\frac{2\pi}{\theta},\theta\geq 2\pi\}\), \(X=S^1\cup R\). By \cite[ Example 4.4 ]{BC08} \(X\) is an atriodic Kelley continuum and not a Kelley remainder. For each positive integer \(n\), define \(R_n=\{(r,\theta,\frac{1}{2n-1}):r=1+\frac{2\pi}{\theta},\theta\in [2\pi,2\pi(n+1)]\}\), \(S_n=\{(r,\theta,\frac{1}{2n}):r=1+\frac{2\pi}{\theta},\theta\in [2\pi,2\pi(n+1)]\}\), \(L_n=\{(1+\frac{1}{n+1},0,z):z\in[\frac{1}{2n},\frac{1}{2n-1}]\}\), \(J_n=\{(2,0,z):z\in[\frac{1}{2n+1},\frac{1}{2n}]\}\), and let
\(S=\bigcup_{n=1}^{\infty} (R_n\cup S_n\cup L_n\cup J_n)\). Finally, define \(Y=S\cup X\), then \(Y\) is a semi-Kelley compactification of the ray \(S\) with remainder \(X\).
\end{proof}

\section{Hereditarily semi-Kelley dendroids}

We prove that hereditarily semi-Kelley dendroids are smooth. We say that a continuum \(X\) is {\it{semi-locally connected at a point}} \(x\in X\) provided that for every open set \(U\) of \(X\) such that \(x\in U\), there exists an open set \(V\) of \(X\) such that \(x\in V\subset U\) and \(X\setminus V\) consists of a finite number of components. We say that \(X\) is {\it{colocally connected at a point}} \(x\in X\) provided that for every open set \(U\) of \(X\) such that \(x\in U\), there exists an open set \(V\) such that \(x\in V\subset U\) and \(X\setminus V\) is connected.

Given a dendroid \(X\), a point \(p\in X\) is {\it{end point}} of \(X\), if \(x\) is an end point of every arc in \(X\) to which \(x\) belongs. Denote by \(E_S(X)\) the set of end points of \(X\) at which \(X\) is semi-locally connected. By \cite[Theorem 3.5]{MK78}, \(E_S(X)\) is equal to the subset of \(X\) consisting of all points at which \(X\) is colocally connected. Let \(AE_S(X)=\cup\{ab\subset X:a,b\in E_S(X)\}\). By \cite[Theorem 4.1]{MK78}, \(AE_S(X)\) is a dense subset of \(X\).

We thank Prof. A. Illanes for a proof of the following lemma.
%\begin{lemma}\label{triod}
%Let \(X\) be a dendroid different from an arc. Then, \(X\) is a 3-od.
%\end{lemma}
%\begin{proof}
%Let \(a,b\in E_S(X)\), with \(a\neq b\), since \(X\) is not an arc, \(X\neq ab\), thus there exists \(c\in E_S(X)\setminus \{a,b\}\). Let \(U_a\), \(U_b\), \(U_c\) be open sets of \(X\) such that \(a\in U_a, b\in U_b, c\in U_c\), the sets \(X\setminus U_a\), \(X\setminus U_b\) and \(X\setminus U_c\) are connected, and \(cl(U_a)\), \(cl(U_b)\), \(cl(U_c)\) are pairwise disjoint. Notice that \(H=(X\setminus U_a)\cap (X\setminus U_b)\cap (X\setminus U_c)\) is a subcontinuum of \(X\) and \(X\setminus H=U_a\cup U_b\cup U_c\). Therefore, \(X\) is a 3-od.
%\end{proof}

%For \(\varepsilon>0\) and a point \(a\in X\) the symbol \(B(\varepsilon,a)\) represents the \(\varepsilon\)-ball around \(a\). %For a non-empty subset \(A\) of \(X\) and \(\varepsilon>0\) we define \(N(\varepsilon,A)=\bigcup_{a\in A} B(\varepsilon,a)\).

\begin{lemma}\label{arcs}
Let \(X\) be a dendroid different from an arc and let \(\{A_n\}_{n=1}^\infty\) be a sequence of arcs in \(X\). Then, \(X\neq\lim A_n\).
\end{lemma}
\begin{proof}
By \cite[Theorem 4.1]{MK78}, \(AE_S(X)\) is a dense subset of \(X\), thus there exist \(a,b\in E_S(X)\), with \(a\neq b\); since \(X\) is not an arc, \(X\neq ab\), hence there exists \(c\in E_S(X)\setminus \{a,b\}\). Let \(U_a\), \(U_b\), \(U_c\) be open sets of \(X\) such that \(a\in U_a, b\in U_b, c\in U_c\), the sets \(X\setminus U_a\), \(X\setminus U_b\), and \(X\setminus U_c\) are connected, and \(cl(U_a)\), \(cl(U_b)\), \(cl(U_c)\) are pairwise disjoint. %Notice that \(H=(X\setminus U_a)\cap (X\setminus U_b)\cap (X\setminus U_c)\) is a subcontinuum of \(X\) and \(X\setminus H=U_a\cup U_b\cup U_c\).

%By Lemma \ref{triod} we have that \(X\) is a 3-od. Assume \(H\) is the core of \(X\), and let \(J,K,L\) be distinct components of \(X\setminus H\); by \cite[Corollary 5.9]{N92}, \(H\cup J\), \(H\cup K\), and \(H\cup L\) are subcontinua of \(X\). Let \(a\in J,\) \(b\in K\), \(c\in L\) and Let \(\varepsilon>0\) be such that \(B(\varepsilon,a)\subset U_a\), \(B(\varepsilon,b)\subset U_b\) and \(B(\varepsilon,c)\subset U_c\).

Assume that \(X=\lim A_n\), thus there exists \(N\in\mathbb{N}\) such that \(A_N\cap U_a\neq\emptyset\), \(A_N\cap U_b\neq\emptyset\), and \(A_N\cap U_c\neq\emptyset\). Let \(a'\in A_N\cap U_a\), \(b'\in A_N\cap U_b\), and \(c'\in A_N\cap U_c\). Assume without loss of generality that \(b'\in a'c'\subset A_N\). Notice that \(a',c'\in  X\setminus U_b\). Since \(X\setminus U_b\) is a subdendroid of \(X\), \(a'c'\subset X\setminus U_b\). Hence, \(b'\in X\setminus U_b\), contradicting \(b'\in U_b\). This finishes the proof.
\end{proof}

In \cite{GV88} E. E. Grace and E. J. Vought introduced the definitions of a Type 1 dendroid and Type 2 dendroid, where they used these conceps to prove that a dendroid is non-smooth if and only if it contains a Type 1 or Type 2 subdendroid, see \cite[Theorem 1]{GV88}. We prove that a hereditarily semi-Kelley dendroid does not contain a Type 1 subdendroid nor a Type 2 subdendroid, thus hereditarily semi-Kelley dendroids are smooth.

A dendroid \(X\) is a {\it{Type 1 dendroid}} provided that there exist points \(p,a\in X\), a sequence \(\{a_i\}_{i=1}^\infty\) in \(X\) converging to \(a\), such that \(X=cl(\bigcup_{i=1}^{\infty}pa_i)\), and the following two properties are satisfied:
\begin{enumerate}
 \item[T1.1.] the sequence \(\{pa_i\}_{i=1}^\infty\) converges to some \(L\in C(X)\),
 \item[T1.2.] there exists an end point \(s\) of \(L\), with \(s\neq a\), and there exists an open set \(U\) of \(X\) such that \(s\in U\) and \(C_s\cap (\bigcup_{i=1}^{\infty}pa_i)=\emptyset\), where \(C_s\) is the component of \(L\cap U\) containing \(s\).
\end{enumerate}

Moreover, if the following property is satisfied:

\begin{enumerate}
 \item
 [T1.3.] if \(i\neq j\), then \(pa_i\cap pa_j\cap cl(U)=\emptyset\),
\end{enumerate}
then \(X\) is called {\it{strong Type 1 dendroid}}.
The point \(p\) is called an {\it emanation point} of \(X\), the point \(a\) is called an {\it absorption point} of \(X\), and the point \(s\) is called a {\it turning point} of \(X\). Notice that \(s\neq p\), moreover, \(s\notin pa\). In the rest of the paper, if we say that \(X=L\cup (\bigcup_{i=1}^{\infty}pa_i)\) is a Type 1 dendroid, we mean that \(p\) is an emanation point, the sequence \(\{a_i\}_{i=1}^\infty\) converges to an absorption point \(a\in X\), and the sequence \(\{pa_i\}_{i=1}^\infty\) converges to \(L\).

%By \cite[Lemma 2]{GV88}, each Type 1 dendroid contains a strong Type 1 dendroid.

\begin{lemma}\label{cocientet1}
Let \(X=L\cup (\bigcup_{i=1}^{\infty}pa_i)\) be a strong Type 1 dendroid with \(a=\lim a_i\). Then, the quotient \(X/pa\) is a strong Type 1 dendroid with an emanation point equal to an absorption point.
\end{lemma}
\begin{proof}
Let \(s\in X\) and let \(U\) be an open set of \(X\) satisfying T1.2 and T1.3; since \(s\notin pa\), we can assume that \(pa\cap U=\emptyset\).
Let \(\pi:X\to X/pa\) be the quotient function. Since \(\pi\) is a monotone function, \(X/pa\) is a dendroid by \cite[Corollary 13.41]{N92}. Let \(\overline{p}=\pi(p)\), \(\overline{a}=\pi(a)\), \(\overline{a}_i=\pi(a_i)\) for each \(i\in\mathbb N\), \(\overline{L}=\pi(L)\), \(\overline{s}=\pi(s)\), and \(\overline{U}=\pi(U)\). Notice that \(\overline{p}=\overline{a},\{\overline{a}_i\}_{i=1}^\infty,\overline{s},\overline{L},\) and \(\overline{U}\) satisfy T1.1, T1.2, and T1.3, thus \(X/pa\) is a strong Type 1 dendroid.
\end{proof}

\begin{lemma}\label{strong1}
Let \(X=L\cup (\bigcup_{i=1}^{\infty}pa_i)\) be a strong Type 1 dendroid such that \(p=\lim a_i\). Then, there exists a subsequence \(\{a_{N_i}\}_{i=1}^{\infty}\) such that \(\lim (pa_{N_i}\cap L)=M\), for some \(M\subsetneq L\), and \(Y=L\cup (\bigcup_{i=1}^{\infty}pa_{N_i})\) is a strong Type 1 dendroid with a turning point \(e\notin M\).
\end{lemma}
\begin{proof}

Let \(s\in L\) be a turning point of \(X\), and let \(U\) be open in \(X\) satisfying T1.2 and T1.3.

Since \(C(X)\) is compact, there exists a subsequence \(\{a_{k_i}\}_{i=1}^\infty\)  such that \(\lim (pa_{k_i}\cap L)=M\), for some \(M\subset L\). We prove that \(M\neq L\); if \(L\) is an arc, then \(M\subset L\setminus C_s\), since \(C_s\) is an open set of \(L\) containing \(s\) and \(pa_i\cap C_s=\emptyset\) for each \(i\in\mathbb N\), thus \(M\neq L\); if \(L\) is not an arc, since \(pa_{k_i}\cap L\) is an arc of \(L\) for each \(i\in \mathbb{N}\), and \(M=\lim (pa_{k_i}\cap L)\), by Lemma \ref{arcs}, \(M\neq L\) .

If \(s\notin M\), then \(e=s\) and \(\{a_{N_i}\}_{i=1}^\infty=\{a_{k_i}\}_{i=1}^\infty\) satisfies the required conclusion. Assume that \(s\in M\). By \cite[Theorem 4.1]{MK78}, \(AE_S(L)\) is a dense subset of \(L\). Hence, there exists \(e\in E_S(L)\) such that \(e\in L\setminus M\). Let \(V\) be an open set of \(X\) such that \(e\in V\) and \(cl(V)\cap M=\emptyset\). %\(L\setminus V\) is connected.Let \(V'\) be an open set of \(X\) such that \(V=V'\cap L\) and \(cl(V')\cap M=\emptyset\).
Since \(X\setminus cl(V)\) is an open set of \(X\) containing \(M\), there exists \(l_1\in\mathbb{N}\) such that \((pa_{k_i}\cap L)\subset X\setminus cl(V)\) and \(pa_{k_i}\cap V\neq \emptyset\) for each \(i\geq l_1\), and observe that \(pa_{k_i}\cap L\cap cl(V)=\emptyset\).

 Notice that there exists \(l_2>l_1\) such that for each \(i\geq l_2\), \(pa_{k_{l_1}}\cap pa_{k_i}\cap\ cl(V)=\emptyset\), otherwise \(pa_{k_{l_1}}\cap L\cap cl(V)\neq\emptyset\), which is a contradiction to the choice of \(l_1\). Similarly, there exists \(l_3>l_2\) such that for each \(i\geq l_3\), \(pa_{k_{l_2}}\cap pa_{k_i}\cap\ cl(V)=\emptyset\). Continuing inductively, we define an increasing sequence \(\{l_j\}_{j=1}^{\infty}\) such that for each \(j\in\mathbb N\), and for each \({i}\geq l_{j+1}\), \(pa_{k_{l_j}}\cap pa_{k_i}\cap\ cl(V)=\emptyset\). Hence, the sequence \(\{a_{k_{l_i}}\}_{i=1}^{\infty}\) satisfies that \(pa_{k_{l_i}}\cap pa_{k_{l_j}}\cap\ cl(V)=\emptyset\) for \(i\neq j\).

Define for each \(i\in\mathbb{N}\), \(N_i=k_{l_i}\) and define \(Y=L\cup (\bigcup_{i=1}^{\infty}pa_{N_i})\). Notice that \(Y=cl(\bigcup_{i=1}^\infty) pa_{N_i}\),
the sequence \(\{a_{N_i}\}_{i=1}^\infty\) is a sequence in \(Y\) converging to \(p\), the sequence \(\{pa_{N_i}\}_{i=1}^\infty\) is a sequence in \(C(Y)\) converging to \(L\). Define \(W=V\cap Y\), observe that \(W\) is an open set of \(Y\) such that \(e\in W\), let \(C_e\) be the component of \(L\cap W\) containing \(e\). We prove that \(W\) satisfies properties T1.2 and T1.3, with \(e\in E_S(L)\setminus M\) a turning point of \(Y\). We prove T1.2; since \((\bigcup_{i=N}^{\infty}pa_{k_i})\cap L\subset X\setminus cl(V)\), we have \((\bigcup_{i=1}^{\infty}pa_{N_i})\cap L\cap W=\emptyset\). Hence, \(C_e\cap (\bigcup_{i=1}^{\infty}pa_{N_i})=\emptyset\). We prove T1.3; by its definition, we have \(pa_{N_i}\cap pa_{N_j}\cap\ cl(W)\subset pa_{N_i}\cap pa_{N_j}\cap\ cl(V)=\emptyset\), for each \(i,j\in\mathbb N\) with \(i\neq j\). We conclude that \(Y\) is a strong Type 1 dendroid with a turning point \(e\notin M\).
\end{proof}

For a subset \(A\subset Y\subset X\) , let \(cl_Y(A)\) and \(bd_Y(A)\) denote the closure of \(A\) in \(Y\) and the boundary of \(A\) in \(Y\), respectively.
\begin{lemma}\label{type1}
 Let \(X=L\cup (\bigcup_{i=1}^{\infty}pa_i)\) be a strong Type 1 dendroid with \(p=\lim a_i\). %\(s\) a turning poing, and let \(M=\lim (pa_i\cap L)\).
  Then, \(X\) is not hereditarily semi-Kelley.
\end{lemma}
\begin{proof}
 For points \(x,y\in X\), we write \(x\leq_p y\) provided that \(x\in py\), and we write \(x<_p y\) provided that \(x\leq_p y\) and \(x\neq y\). For \(x<_p y\), define \((x,y]=\{t\in xy:x<_p t\leq_p y\}=xy\setminus\{x\}\). If \(A\subset px\), for some \(x\in X\), \(\inf(A)\) means the greatest lower bound of \(A\) in \(px\) with the order \(\leq _p\).

By Lemma \ref{strong1}, there exists a subsequence \(\{a_{N_i}\}_{i=1}^{\infty}\) such that \(\lim (pa_{N_i}\cap L)=M\), for some \(M\subsetneq L\), and \(Y=L\cup (\bigcup_{i=1}^{\infty}pa_{N_i})\) is a strong Type 1 dendroid with a turning point \(s\notin M\). For each \(i\in\mathbb{N}\), define \(b_i=a_{N_i}\), thus \(Y=L\cup (\bigcup_{i=1}^{\infty}pb_i)\). Let \(U\) be an open set of \(Y\) such that \(s\in U\) and \(U\) satisfies T1.2 and T1.3; since \(s\notin M\), we can choose \(U\) such that \(M\cap cl_Y(U)=\emptyset\). Notice that \(Y\setminus cl_Y(U)\) is an open set of \(Y\) containing \(M\) and \(s\in\lim pb_i\), hence, there exists \(N\in\mathbb{N}\) such that for each \(i\geq N\), \((pb_i\cap L)\subset Y\setminus cl_Y(U)\) and \(pb_i\cap U\neq\emptyset\).

%For each \(i\in\mathbb{N}\), we consider the natural order of \(pa_i\), where \(p\) is the least element.%Notice that an arc is not a Type 1 dendroid, hence by Theorem \ref{arcs}, \(L\neq X\).
%Taking a subsequence if necessary, we may assume that \(\{pa_i\cap L\}_{i=1}^\infty\) is a convergent sequence in \(C(X)\), and define \(M=\lim (pa_i\cap L)\), notice that \(M\) is a subdendroid of \(L\).

%\textbf{Claim 1.} \(M\subsetneq L\).

%By \cite[Theorem 4.1]{MK78}, \(AE_s(L)\) is a dense subset of \(L\). Hence, there exists \(e\in E_s(L)\) such that \(e\in L\setminus M\). Let \(V\) be an open set of \(L\) such that \(e\in V\) and \(cl(V)\cap M=\emptyset\). %\(L\setminus V\) is connected.
%Let \(V'\) be an open set of \(X\) such that \(V=V'\cap L\) and \(cl(V')\cap M=\emptyset\). Define \(Z=L\cup (\bigcup_{i=N+1}^{\infty}pa_i)\).

%\textbf{Claim 2.} \(Z\) is a Type 1 dendroid with turning point \(e\).

%{\it Proof of Claim 2.} Notice that \(p,a\in Z\),
%the sequence \(\{a_i\}_{i=N+1}^\infty\) is a sequence in \(Z\) converging to \(a\), and the sequence \(\{pa_i\}_{i=N+1}^\infty\) is a sequence in \(C(Z)\) converging to \(L\). Define \(W=V'\cap Z\), observe that \(W\) is an open set of \(Z\) such that \(e\in W\). Since \((\bigcup_{i=N+1}^{\infty}pa_i)\cap L\subset X\setminus cl(V')\), thus \((\bigcup_{i=N+1}^{\infty}pa_i)\cap L\cap W=\emptyset\). Hence, \(C_e\cap (\bigcup_{i=N+1}^{\infty}pa_i)=\emptyset\), where \(C_e\) is the component of \(L\cap W\) containing \(e\). Thus, \(Z\) is a Type 1 dendroid with \(e\) as its turning point. Claim 2 is proved.

For each \(i\geq N\), define \(u_i=\inf (pb_i\cap bd_Y(U))\), and notice that \(u_i\in bd_Y(U)\).  Taking subsequence if necessary, we may assume that the sequence \(\{u_i\}_{i=N}^\infty\) is convergent in \(Y\) and the sequence \(\{u_ib_i\}_{i=N}^\infty\) is convergent in \(C(Y)\).

\textbf{Claim 1.} For each \(i\geq N\), \(u_ib_i\) is an arc without ramification points of \(Y\).

{\it Proof of Claim 1.}

Assume \(v\in u_ib_i\) is a ramification point of \(Y\). Thus there exist \(w\in Y\setminus pb_i\) such that \(pb_i\cap pw=pv\). If \(w\in pb_j\) for some \(i\neq j\), then \(u_i\in pb_i\cap pb_j\cap cl_Y(U)\), which contradicts T1.3.
Hence, \(w\in L\), since \(pw\subset L\), \(u_i\in pv=pb_i\cap pw\subset pb_i\cap L\subset Y\setminus cl_Y(U)\), which contradicts the definition of \(u_i\). Claim 1 is proved.

Define \(A=\lim u_ib_i\), note that \(A\subset L\). By a similar argument as in \cite[Theorem 6.1]{I21}, \(A\) is an arc.

Let \(\{s_i\}_{i=1}^\infty\) be a sequence in \(Y\) converging to \(s\); without loss of generality assume that \(s_i\in pb_i\cap U\) for each \(i\geq N\). Since \(s_i\in pb_i\cap U\), we have that \(s_i\in u_ib_i\) for each \(i\geq N\), thus \(s\in\lim u_ib_i=A\). Moreover, since \(\lim b_i=p\), \(p\in \lim u_ib_i=A\). Hence, \(ps\subset A\) and \(s\) is an end point of \(A\).

 \textbf{Claim 2.} \(L \cap U=A\cap U\).

{\it Proof of Claim 2.}

Since \(A\subset L\), then \(A\cap U\subset L\cap U\).
Let \(w\in L\cap U\), then \(w\in\lim pb_i\), for which there exists a sequence \(\{w_i\}_{i=1}^\infty\) such that \(w_i\in pb_i\cap U\) and \(\lim w_i=w\). Hence, \(w_i\in u_ib_i\) and \(w\in \lim u_ib_i=A\). Claim 2 is proved.

Let \(C_s\) be the component of \(L\cap U\) containing \(s\). By Claim 2, \(C_s\) is the component of \(A\cap U\) containing \(s\). Since \(A\) is an arc with an end point \(s\) and \(p\in A\setminus U\), we have that there exists \(x\in ps\) such that \(C_s=(x,s]\). Let \(q\in A\) be the end point of \(A\) with \(q\neq s\), notice that \(A\setminus (x,s]= px\cup pq\). Since \(px\cup pq\) is a closed set of \(Y\) we have that \(A\setminus (x,s]\) is a closed set of \(A\). Therefore, \((x,s]\) is an open set of \(A\). Let \(W\) be open set of \(Y\) such that \(W\cap A=(x,s]\). Notice that \((W\cap U)\cap L\) is an open set of \(L\) and \((W\cap U)\cap L=W\cap(U\cap L)=W\cap (U\cap A)=(W\cap A)\cap U=(x,s]\cap U=(x,s]\), thus  \((x,s]\) is an open set of \(L\). Notice that for each \(y\in (x,s]\setminus\{s\}\), \(L\setminus(y,s]=(L\setminus (x,s])\cup xy\) is a closed set of \(L\), hence, \((y,s]\) is an open set  of \(L\).

Let \(y,z\in (x,s]\) be such that \(x<_p y<_p z<_p s\). Since \((y,s]\) and \((z,s]\) are open sets of \(L\), there are \(U_y\) and \(U_z\) open sets of \(Y\) such that \(U_y\cap L=(y,s]\), \(U_z\cap L=(z,s]\), and \(cl_Y(U_z)\subset U_y\subset cl_Y(U_y)\subset U\cap W\). Define for each \(i\geq N\), \(y_i=\inf(ps_i\cap bd_Y(U_y))\) and  \(z_i=\inf(s_ib_i\cap bd_Y(U_z))\), observe that \(y_i\in bd_Y(U_y)\) and \(z_i\in bd_Y(U_z)\). Taking subsequences if necessary, we may assume that the sequences \(\{y_i\}_{i=N}^\infty\) and \(\{z_i\}_{i=N}^\infty\) are convergent in \(Y\) and the sequences \(\{py_i\}_{i=N}^\infty\) and \(\{pz_i\}_{i=N}^\infty\)  are convergent in \(C(Y)\). Let  \(\lim u_i=u'\), \(\lim y_i=y'\), and \(\lim z_i=z'\).

\textbf{Claim 3.} \(u'\notin (x,s] \), \(y'\in (x,y]\), and \(z'\in (y,z]\).

{\it Proof of Claim 3.}

We prove that \(y'\in (x,y]\). Since for each \(i\geq N\), \(y_i\in pb_i\cap bd_Y(U_y)\), we have that \(y'\in L\cap bd_Y(U_y)\subset L\cap cl_Y(U_y)\subset L\cap U\cap W=(x,s]\). Assume \(y'\in (y,s]\), as \((y,s]\subset U_y\), there exists \(k\geq N\) such that \(y_k\in U_y\), which contradicts that  \(y_k\in bd_Y(U_y)\). Thus \(y'\in (x,y]\). A similar argument shows that \(u'\notin (x,s]\) and \(z'\in (y,z]\). Claim 3 is proved.

Define \(Z=(\bigcup_{i=1}^{\infty}py_{2i+N})\cup(\bigcup_{i=1}^{\infty}pz_{2i+1+N})\cup L\); since \(\lim py_{2i+N}\subset L\) and \(\lim pz_{2i+1+N}\subset L\), we have that \(Z\) is a subcontinuum of \(Y\).

\begin{figure}[h]
  \centering
  \includegraphics[width=9cm]{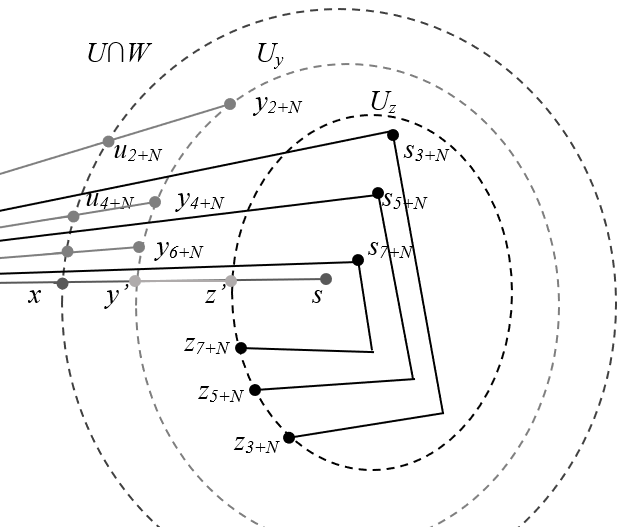}
  \caption{The dendroid \(Z\) around \(s\).}\label{}
\end{figure}

 We will prove that \(Z\) is not semi-Kelley. Notice that \(y',z'\in Z\), \(\lim y_{2i+N}=y'\), \(\lim z_{2i+1+N}=z'\), and \(y'z'\subset (x,z]\subset Z\) is irreducible between \(y'\) and \(z'\). Let \(\{A_i\}_{i=1}^\infty\) be any convergent sequence in \(C(Z)\) such that \(y_{2i+N}\in A_i\) for each \(i\in\mathbb N\).

\textbf{Claim 4.} The sequence \(\{A_i\}_{i=1}^\infty\) does not converge to \(y'z'\).

{\it Proof of Claim 4.}

Notice that \(A_i\subset u_{2i+N}y_{2i+N}\subset Z\setminus U_y\) for an infinite number of \(i's\) or \(u_{2i+N}y_{2i+N}\subset A_i\) for an infinite number of \(i's\). In the first case \(\lim A_i\subset Z\setminus U_y\), since \(z'\notin Z\setminus U_y\) we have that \(\lim A_i\neq y'z'\).  In the second case \(u'\in\lim A_i\), since \(u'\notin (x,z]\) and \(y'z'\subset (x,z]\), we have that \(\lim A_i\neq y'z'\). Claim 4 is proved.

Let \(\{B_i\}_{i=1}^\infty\) be any convergent sequence in \(C(Y)\) such that \(z_{2i+1+N}\in B_i\) for each \(i\in\mathbb N\).

\textbf{Claim 5.} The sequence \(\{B_i\}_{i=1}^\infty\) does not converge to \(y'z'\).

{\it Proof of Claim 5.}

Notice that \(B_i\subset s_{2i+1+N}z_{2i+1+N}\subset cl_Y(U_z)\) for an infinite number of \(i's\) or \(s_{2i+1+N}z_{2i+1+N}\subset B_i\) for an infinite number of \(i's\). In the first case \(\lim B_i\subset cl_Y(U_z)\subset U_y\); since \(y'\notin U_y\) we have that \(\lim B_i\neq y'z'\).  In the second case \(s\in\lim B_i\); since \(s\notin y'z'\)  we have that \(\lim B_i\neq y'z'\). Claim 5 is proved.

By Claim 4, Claim 5, and Theorem \ref{equivalencia}, \(Z\) is not semi-Kelley, thus \(X\) is not hereditarily semi-Kelley. The proof is complete.
\end{proof}

\begin{lemma}\label{type11}
Let \(X\) be a hereditarily semi-Kelley dendroid. Then, \(X\) does not contain a Type 1 subdendroid.
\end{lemma}
\begin{proof}
Assume that \(X\) contains a Type 1 subdendroid \(Y\). By \cite[Lemma 2]{GV88}, each Type 1 dendroid contains a strong Type 1 dendroid. Let \(Z=L\cup (\bigcup_{i=1}^{\infty}pa_i)\subset Y\subset X\)  be a strong Type 1 dendroid with \(a=\lim a_i\). If \(p=a\), then by Lemma \ref{type1}, \(Z\) is not hereditarily semi-Kelley, which is a contradiction. If \(p\neq a\), then by Lemma \ref {cocientet1}, \(Z/pa\) is a strong Type 1 subdendroid of \(X/pa\).

\textbf{Claim.} \(X/pa\) is hereditarily semi-Kelley.

{\it Proof of Claim.}

Let \(\pi:X\to X/pa\) be the quotient function. Let \(W\in C(X/pa)\),  since  \(\pi\) is a monotone map we have that \(\pi^{-1}(W)\in C(X)\). Hence, \(\pi^{-1}(W)\) is semi-Kelley. Notice that \(\pi(p)\in W\) or  \(\pi(p)\notin W\). In the first case  \(pa\subset\pi^{-1}(W)\) and \(W\) is homeomorphic to \(\pi^{-1}(W)/pa\). Therefore, by \cite[Theorem 8]{FP19} \(\pi^{-1}(W)/pa\) is semi-Kelley, and so \(W\) is semi-Kelley.  In  the second case \(W\) is homeomorphic to \(\pi^{-1}(W)\). Hence, \(W\) is semi-Kelley. Therefore, \(X/pa\) is hereditarily semi-Kelley. The Claim is proved.

  Notice that \(Z/pa\) is a strong Type 1 subdendroid of \(X/pa\), with an emanation point \(\pi(p)=\pi(a)\). Then, by Lemma \ref{type1}, \(Z/pa\) is not hereditarily semi-Kelley, which contradicts the Claim. The proof is complete.
 \end{proof}

Recall that a dendroid \(X\) is a {\it{Type 2 dendroid}} provided that there exist points \(s,a,b,t\in X\), sequences \(\{a_i\}_{i=1}^\infty\), \(\{b_i\}_{i=1}^\infty\) in \(X\) such that \(X=(\bigcup_{i=1}^{\infty}sa_i)\cup sa\cup ab\cup bt\cup(\bigcup_{i=1}^{\infty}tb_i)\), and the following properties are satisfied:
\begin{enumerate}
 \item[T2.1.] \(\lim a_i=a\) and \(\lim b_i=b\),
 \item[T2.2.] \(\lim sa_i=sa\) and \(\lim tb_i=tb\),
 \item[T2.3.] \(\lim [{\text{diam}}(sa_i\cap st)]=0=\lim [{\text{diam}}(tb_i\cap st)]\), and
 \item[T2.4.] the points \(a,b\in st\) and \(s<a<b<t\) in the natural order from \(s\) to \(t\).
\end{enumerate}

\begin{lemma}\label{type2}
Let \(X\) be a hereditarily semi-Kelley dendroid. Then, \(X\) does not contain a Type 2 subdendroid.
\end{lemma}
\begin{proof}
Assume that \(X\) contains a Type 2 subdendroid \(Y\). Let \(s,a,b,t\in Y\) and let \(\{a_i\}_{i=1}^\infty, \{b_i\}_{i=1}^\infty\) be sequences in \(Y\) such that \(Y=(\bigcup_{i=1}^{\infty}sa_i)\cup sa\cup ab\cup bt\cup(\bigcup_{i=1}^{\infty}tb_i)\) and T2.1-T2.4 are satisfied. Since \(sa\cap bt=\emptyset\), taking a subsequence if necessary, we can assume that \(\bigcup_{i=1}^\infty sa_i\) and \(\bigcup_{i=1}^\infty tb_i\) are disjoint, and \(a_i,b_i\notin st\) for each \(i\in\mathbb N\). Since \(X\) is hereditarily semi-Kelley, \(Y\) is a semi-Kelley. As the sequences \(\{a_i\}_{i=1}^\infty\) and \(\{b_i\}_{i=1}^\infty\) in \(Y\) converge to \(a\) and \(b\), respectively, \(ab\) is irreducible between \(a\) and \(b\), by Theorem \ref{equivalencia}, without loss of generality, there exists a sequence \(\{A_i\}_{i=1}^\infty\) in \(C(Y)\) converging to \(ab\) such that \(a_i\in A_i\) for each \(i\in\mathbb N\). Let \(c\in ab\setminus\{a,b\}\), notice that \(Y=(sc\cup\bigcup_{i=1}^{\infty}sa_i)\cup (ct\cup\bigcup_{i=1}^{\infty}tb_i)\) and \(\{c\}=(sc\cup\bigcup_{i=1}^{\infty}sa_i)\cap (ct\cup\bigcup_{i=1}^{\infty}tb_i)\), thus \(Y\setminus\{c\}\) is not connected, moreover, \(c\) separates \(a\) and \(b\) in \(Y\). Since \(\lim A_i=ab\), and \(c\) separates \(a\) and \(b\) in \(Y\), there exists \(N\in\mathbb{N}\) such that \(c\in A_i\) for each \(i\geq N\); hence, \(a_ic\subset A_i\) for each \(i\geq N\). Define \(s_i=\inf(a_ic\cap st)\) with the order \(\leq _s\) of the arc \(st\), defined by \(x\leq_s y\) provided that \(x\in sy\). Notice that \(s_i\in sa_i\cap st\) and \(s_i\in A_i\), for each \(i\geq N\); since \(\lim [{\text{diam}}(sa_i\cap st)]=0\), the sequence \(\{s_i\}_{i=1}^\infty\) converges to \(s\), thus \(s\in ab\), contradicting that \(s<a<b\) in the natural order from \(s\) to \(t\). The proof is complete.
\end{proof}

\begin{theorem}
Let \(X\) be a hereditarily semi-Kelley dendroid. Then, \(X\) is smooth.
\end{theorem}
\begin{proof}
By Lemma \ref{type11} and Lemma \ref{type2}, \(X\) does not contain a Type 1 subdendroid nor a Type 2 subdendroid. By \cite[Theorem 1]{GV88}, \(X\) is a smooth dendroid.
\end{proof}

The converse of the previous Corollary is not true by \cite[Example 13]{C03}. The following unsolved questions arose naturally in our study of semi-Kelley dendroids.
\begin{question}
Let \(X\) be a dendroid. Is it true that there exists a semi-Kelley dendroid \(Y\) such that \(X\subset Y\)?
\end{question}

\begin{question}
Let \(X\) be a (semi-)Kelley continuum, let \(Y\in C(X)\) and let \(U\) be an open set of \(X\) such that \(Y\subset U\). Is it true that there exists a (semi-)Kelley continuum \(Z\in C(X)\) such that \(Y\subset Z\subset U\)?
\end{question}

The main unsolved question about semi-Kelley continua is the following, originally stated in \cite[Question 5.16]{CC98}.
\begin{question}
Let \(X\) be a semi-Kelley continuum. Is it true that \(C(X)\) is contractible?
\end{question}

\end{document}